\documentclass{amsart}
\usepackage{amssymb,amsmath,mathrsfs,pgf,tikz}
\usepackage{hyperref}
\hypersetup{
  colorlinks  = true, 
  urlcolor    = black, 
  linkcolor   = black, 
  citecolor   = black 
}

\newcommand{\TT}{\mathbb{T}}
\newcommand{\PP}{\mathbb{P}}
\newcommand{\NN}{\mathbb{N}}
\newcommand{\ZZ}{\mathbb{Z}}

\newcommand{\G}{\mathscr{G}}
\renewcommand{\P}{\mathcal{P}}

\newcommand{\K}{\mathcal{K}}

\newcommand{\x}{\mathbf{x}}
\newcommand{\y}{\mathbf{y}}
\renewcommand{\t}{\mathbf{t}}
\newcommand{\reg}{\operatorname{reg}}
\newcommand{\set}[1]{\left \{ #1 \right \}}

\newcommand{\ts}{\textstyle}
\renewcommand{\ss}{\scriptstyle}

\newtheorem{theorem}{Theorem}[section]

\newtheorem{cor}[theorem]{Corollary}
\newtheorem{prop}[theorem]{Proposition}
\newtheorem*{claim*}{Claim}

\theoremstyle{definition}
\newtheorem{definition}[theorem]{Definition}
\newtheorem{rmk}[theorem]{Remark}

\newtheorem{examp}[theorem]{Example}

\begin{document}

\title[Joins, Ears and Castelnuovo--Mumford regularity]%
{Joins, Ears and Castelnuovo--Mumford regularity}

\author{J.~Neves}
\address{\emph{J.~Neves}: CMUC, Department of Mathematics, University of Coimbra, 3001-501 Coimbra, Portugal.}
\email{neves@mat.uc.pt}
\thanks{This work was partially supported by the Centre for Mathematics of the University of Coimbra -- UID/MAT/00324/2019, 
funded by the Portuguese Government through FCT/MEC and co-funded by the European Regional Development Fund through the Partnership Agreement PT2020.}

\author{M.~Vaz Pinto}
\address{\emph{M.~Vaz Pinto}: CAMGSD, Departamento de Matem\'atica, Instituto Superior T\'ecnico, Universidade de Lisboa, Avenida Rovisco Pais, 1, 1049-001 Lisboa, Portugal.}
\email{vazpinto@math.tecnico.ulisboa.pt}
\thanks{The second author was partially supported by the Center for Mathematical Analysis, Geometry and Dynamical Systems of Instituto Superior T\'ecnico, Universidade de Lisboa, funded by FCT/Portugal through UID/MAT/04459/2013}

\author{R.~H.~Villarreal}
\address{\emph{R.~H.~Villarreal}: 
Departamento de Matem\'aticas, Centro de Investigaci\'on y de Estudios Avanzados del IPN, Apartado Postal 14--740, 07000 Mexico City, D.F., Mexico.}
\email{vila@math.cinvestav.mx}


\begin{abstract}
We introduce a new class of polynomial ideals associated to a simple graph, $G$. 
Let $K[E_G]$ be the polynomial ring on the edges 
of $G$ and $K[V_G]$ the polynomial ring on the vertices of $G$. We associate to $G$  
an ideal, $I(X_G)$, defined as the preimage of $(x_i^2-x_j^2 : i,j\in V_G)\subseteq K[V_G]$ by 
the map $K[E_G]\to K[V_G]$ which sends a variable, $t_e$, associated 
to an edge \mbox{$e=\set{i,j}$,} to the product $x_ix_j$ of the variables associated to its vertices.
We show that $K[E_G]/I(X_G)$ is a one-dimensional, Cohen-Macaulay, graded ring, that 
$I(X_G)$ is a binomial ideal and that, with respect to a fixed monomial order, its initial 
ideal has a generating set independent of the field $K$. We focus on the Castelnuovo--Mumford
regularity of $I(X_G)$ providing the following sharp upper and lower bounds:
$$
\mu(G) \leq \reg I(X_G) \leq |V_G|-b_0(G)+1,
$$
where $\mu(G)$ is the maximum vertex join number of the graph and $b_0(G)$ is the number of its connected components.
We show that the lower bound is attained for a  
bipartite graph and use this to derive a new combinatorial result on the number of even length ears of 
nested ear decomposition. 
\end{abstract}
\maketitle

\section{Introduction}
The study of polynomial ideals associated to combinatorial structures, ex\-plor\-ing
relations between algebraic and combinatorial invariants, has been a source for many new results. 
In the case of graphs, these ideals include, but are not limited to, 
the \emph{toric ideal}, the \emph{edge ideal} and  \emph{binomial edge ideal}, and, in this framework, 
one of the algebraic invariants that has been the object of 
growing interest is the Castelnuovo--Mumford regularity. 
Bounds for the regularity of the toric ideal have been obtained in \cite{BiO'KVT17,HaBeO'K19}. 
These bounds involve the number and sizes of families of disjoint induced complete bipartite subgraphs of the graph.
By \cite{Ka06}, the regularity of the edge ideal is bounded below by the induced matching number plus $1$, 
and, by \cite{Wo14}, it is bounded above by the co-chordal number plus $1$.
Several classes of graphs for which regularity of the edge ideal attains one, or both, of these bounds have been studied 
--- see \cite{BeHaTr15,  HaVT08, HiHiKiOK15, HiHiKiTs16, Hig17, KhMo14, MaMoCrRiTeYa11, VT09, Wo14}.  
Refinements of these bounds were recently obtained in \cite{SFYa18}.
In \cite{MaMu13}, it is shown that the regularity of the binomial edge ideal is bounded below by the 
length of the longest induced path of the graph plus $1$ and above by the number of its vertices. It is conjectured that the number of maximal
cliques of the graph plus $1$ is an upper bound for the regularity of the binomial edge ideal --- see \cite{KiSM13,KiSM16}.
See also \cite{EnZa15, KiSM16, KiSM18} for values of the regularity of the binomial edge ideal for specific classes of graphs.
\smallskip

In this article we define a new class of polynomial ideals associated to graphs and study 
their Castelnuovo--Mumford regularity. Let $G$ be a simple graph on a finite vertex set, 
$V_G\subseteq \NN$, without isolated vertices. Let its edge set be denoted by $E_G$, 
let $K$ be a field and let $K[V_G]$ and $K[E_G]$ denote the polynomial rings 
$$
K[V_G]=K[x_i : i\in V_G],\quad K[E_G] = K[t_e : e\in E_G],
$$
associated to the vertex and edge set, respectively.
Let $\theta \colon K[E_G]\to K[V_G]$ be the ring homomorphism defined by
$t_e\stackrel{}{\mapsto} x_ix_j$, for every $e=\set{i,j}\in E_G$.
In particular, recall, the toric subring of $G$ is the image of $\theta$ and 
the toric ideal of $G$ is $\ker \theta$. 
\begin{definition}\label{def: the ideals}
Let $I(X_G)\subseteq K[E_G]$ be given by $I(X_G)=\theta^{-1}(x_i^2-x_j^2 : i,j\in V_G)$.
\end{definition}

\noindent
Since $\theta$ is graded and the ideal 
$(x_i^2-x_j^2 : i,j\in V_G)$ 
is homogeneous, $I(X_G)$ is also a homogeneous ideal. 
In fact, we will show that $I(X_G)$ is generated by homogeneous binomials and that 
with respect to a fixed monomial order its initial ideal has a generating set independent 
of the field. The main results of this article reveal a strong connection between the
Castelnuovo--Mumford regularity of $I(X_G)$ and the maximum vertex join number of $G$.

\begin{definition}{\cite{SoZa93,Fr93}}\label{def: max vertex join number}
The maximum vertex join number of $G$  
is the ma\-xi\-mum cardinality of $J\subseteq E_G$ satisfying 
$|J\cap E_C|\leq |E_C|/2$, for every circuit $C$ in $G$.
\end{definition}

\noindent
Following \cite{Fr93}, we will denote the maximum vertex join number by $\mu(G)$. 
Let $b_0(G)$ denote the number of connected components of $G$. Then, 
by Theorem~\ref{thm: reg bigger than join} and Proposition~\ref{prop: general upper bound for reg}, proved in this article, 
the following bounds hold, for any graph:
$$
\mu(G)\leq \reg I(X_G) \leq |V_G|-b_0(G)+1.
$$

\noindent
Moreover, by Theorem~\ref{thm: reg for bipartites is max join}, if $G$ is bipartite, then $\reg I(X_G) = \mu(G)$.
This relation and the results of \cite{Ne}, on the regularity of the vanishing ideal over a graph endowed
with \emph{nested} ear decomposition, yield a new combinatorial result (see Corollary~\ref{cor: new combinatorial corollary}).
\smallskip

The motivation for the definition of $I(X_G)$ comes from the 
notion of vanishing ideal over a graph for a finite field, 
introduced by Renteria, Simis and Villarreal in \cite{ReSiVi11}. We will see,
in Proposition~\ref{prop: Link to vanishing ideals}, that the two ideals coincide 
when $K=\ZZ_3$. 
It is this relation and the existence of a set of generators of the initial ideal 
of $I(X_G)$ independent of the field that allow transferring to $I(X_G)$ 
the known properties and values of the regularity 
of the vanishing ideal over a graph.
\smallskip

This article is organized as follows. In Section~\ref{sec: the properties} we will study the basic properties of 
$I(X_G)$. We start by showing that $K[E_G]/I(X_G)$ is a one-dimensional Cohen--Macaulay graded ring 
(Proposition~\ref{prop: one-dimensional and Cohen-Macaulay}). We then show that $I(X_G)$ is a binomial ideal
(Proposition~\ref{prop: general result to get binomial ideal} and Corollary~\ref{cor: the ideal is binomial})
and we characterize binomials in $I(X_G)$ in terms of associated subgraphs of 
$G$ (Proposition~\ref{prop: even degrees}). Next we prove that, 
with respect to a given monomial order, the initial ideal of $I(X_G)$ has a generating set
independent of the field (Proposition~\ref{prop: initial ideal is independent of the field}). We then show 
that $I(X_G)$ coincides with the vanishing ideal over the graph when $K=\ZZ_3$ 
(Proposition~\ref{prop: Link to vanishing ideals})
and give a first application of these two results 
to the computation of the degree of the ideal (Proposition~\ref{prop: degree of the ideal}).
In Section~\ref{sec: regularity}, using the fact that the regularity of $I(X_G)$ 
is independent of the field, we transfer from the context of the vanishing ideal over the graph known 
properties and values of the regularity (Proposition~\ref{prop: known properties and values}). We also 
describe two useful results in our approach to the computation of the regularity 
(Propositions~\ref{prop: computing reg by reducing to Artinian quotient} and \ref{prop: special replacement}).
In Section~\ref{sec: joins and ears} we describe the connection between $\reg I(X_G)$ and the maximum vertex join number,
first establishing upper and lower bounds that hold for any graph (Theorem~\ref{thm: reg bigger than join} and 
Proposition~\ref{prop: general upper bound for reg}) and then pro\-ving equality between 
the regularity and the lower bound, $\mu(G)$, in  the bipartite case 
(Theorem~\ref{thm: reg for bipartites is max join}). We then use this theorem to deduce 
a new combinatorial result related to the number of even length ears of \emph{nested} ear decompositions 
of a bipartite graph (Corollary~\ref{cor: new combinatorial corollary}).

\section{The ideals}\label{sec: the properties}

\subsection{Assumptions and notation}
The graphs considered in this work are finite simple graphs without isolated vertices.
$K$ is any field and, as in the introduction, 
$K[V_G]$ and $K[E_G]$ will denote the polynomial rings on the vertex and 
edge sets of the graph, respectively. 
Given an edge $e=\set{i,j}$, we will also use $t_{ij}$ as an alternative notation to
$t_e$. 
Monomials in $K[V_G]$ and $K[E_G]$ will
be denoted using the multi-index notation. Namely, given 
$\alpha\in\NN^{V_G}$ and $\beta \in \NN^{E_G}$, 
the notations $\x^\alpha$ and $\t^\beta$ shall stand for the monomials
$$
\ts \x^\alpha = \prod_{i\in V_G} x_i^{\alpha(i)} \;\; \text{and}\;\;\; \t^\beta = \prod_{e\in E_G} t_e^{\beta(e)},
$$
respectively.

\subsection{The Cohen--Macaulay property}\label{subsec: CM property}
For the sake of clarity and also for later use, we begin by dealing with the case when $G$ is a single edge.
Assume, without loss of generality, that $V_G=\set{1,2}$ and $E_G=\set{\set{1,2}}$. The map 
$\theta \colon K[E_G]\to K[V_G]$ is then defined by sending 
the unique variable in the domain, $t_{12}$, to the product $x_1x_2\in K[V_G]$. 
Let $f\in K[E_G]$, which we write as:
$$
f = a_0 + a_1t_{12} +\cdots a_d t_{12}^d,
$$
for some $a_0,\dots,a_d\in K$ and $d\in \NN$. If $\theta(f)\in (x_1^2-x_2^2)$ then, setting 
$x_2=x_1$ in $\theta(f)$, we deduce that:
$$
a_0 + a_1x_1^2+a_2x_1^4 + \cdots + a_d x_1^{2d} = 0, 
$$
which implies that $a_0=\cdots =a_d = 0$, i.e. that, $f=0$. Therefore, if $G$ consists of a single edge, $I(X_G) = (0)$. 
In this situation $K[E_G]/I(X_G) \simeq K[t_{12}]$, which is clearly a one-dimensional Cohen--Macaulay graded ring. 
\smallskip

Taking now $G$ a general graph, if $\set{i,j},\set{k,\ell}\in E_G$ are two edges in $G$, one can easily see  
that $t_{ij}^2-t_{k\ell}^2 \in I(X_G)$. Indeed,
$$
\theta(t_{ij}^2-t_{k\ell}^2) = x_i^2x_j^2 - x_k^2x_\ell^2 = (x_i^2-x_k^2)x_j^2 + x_k^2(x_j^2 -x_\ell^2).
$$
Therefore $(t^2_{ij}-t^2_{k\ell} :\, \ss \set{i,j},\set{k,\ell} \ts \in E_G) \subseteq I(X_G)$.

\begin{prop}\label{prop: one-dimensional and Cohen-Macaulay}
$K[E_G]/I(X_G)$ is one-dimensional and Cohen--Macaulay. 
\end{prop}

\begin{proof}
We may assume that $|E_G|>1$.  
Then, in view of the above, the zero set of $(I(X_G),t_{ij})$ in affine space is the singleton $\set{(0,\dots,0)}$, for any $\set{i,j}\in E_G$. 
By \cite[Proposition~8.3.22]{monalg}, we conclude that
$$
\operatorname{ht} I(X_G) = |E_G|-1.
$$
Hence $K[E_G]/I(X_G)$ is a one-dimensional graded ring. 
To show that $K[E_G]/I(X_G)$ is Cohen--Macaulay we will show that it contains a regular element. 
Consider an element of the form $\t^\delta+I(X_G)$, with $\delta\in \NN^{E_G}\setminus  0$. Let us show that this element
is regular. It suffices to consider the case $\t^\delta = t_{ij}$ for some $\set{i,j}\in E_G$.
Without loss of generality, let this edge be $\set{1,2}$.
By Definition~\ref{def: the ideals}, showing that $t_{12}$ is a regular element  
of $K[E_G]/I(X_G)$ can be achieved by showing that $x_1x_2$ is a regular element
of $K[V_G]/(x_i^2-x_j^2 : i,j\in V_G)$. To this end, by symmetry, it is enough to prove that 
$x_1$ is a regular element of $K[V_G]/(x_i^2-x_j^2 : i,j\in V_G)$. Assume that $g\in K[V_G]$ is such that 
\begin{equation}\label{eq: L302}
x_1g \in (x_i^2-x_j^2 : i,j\in V_G).
\end{equation}
Let $k\in V_G$ be a vertex different from $1$. Then 
$x_i^2-x_k^2$, when $i$ varies in $V_G\setminus \set{k}$, yields a Gr\"obner basis for the ideal  
$(x_i^2-x_j^2 : i,j\in V_G)$, with respect to a monomial order where $x_k$ is the least variable. 
Since we want to show that $g$ belongs to the ideal $(x_i^2-x_j^2 : i,j\in V_G)$ we may assume 
that no term of $g$ is divisible by $x_i^2$, for any $i\in V_G\setminus \set{k}$, and aim to show that 
$g=0$. Assume that $g\not = 0$. Then, from \eqref{eq: L302}, we deduce 
that at least one term of $g$ must be divisible by $x_1$. If $c_\delta \x^\delta$, where $\delta\in \NN^{V_G}$ and 
$c_\delta \in K$, is a term of $g$ divisible by $x_1$ (and not by $x_1^2$) 
then the division of $x_1(c_\delta \x^\delta)$ by $x_1^2-x_k^2$ yields 
$$
\ts x_1(c_\delta \x^\delta) = c_\delta \frac{\x^\delta}{x_1} (x_1^2-x_k^2) + c_\delta\frac{\x^\delta}{x_1} x_k^2.
$$
If $\x^\delta$, where $\delta$ varies in some set $\Delta \subseteq \NN^{V_G}$, are the supporting monomials for 
terms of $g$ divisible by $x_1$ (and not by $x_1^2$) and $\x^\gamma$, where $\gamma$ varies in $\Gamma \subseteq \NN^{V_G}$, are those 
supporting terms of $g$ that are not divisible by $x_1$, then it is clear that 
$$
\ts \{\frac{\x^\delta}{x_1}x_k^2 : \delta \in \Delta \} \cup \set{ x_1\x^\gamma : \gamma \in \Gamma}
$$
remains a linearly independent set of monomials. This implies that the remainder 
of $x_1g$ by the division by the Gr\"obner basis of the ideal $(x_i^2-x_j^2 : i,j\in V_G)$ is not zero, contradicting
\eqref{eq: L302}. Hence, we must have $g=0$.
\end{proof}

\subsection{The binomial property} 
To prove that $I(X_G)$ is a binomial ideal, 
we shall use the next proposition, the proof of 
which follows closely the proof of \cite[Theo\-rem~2.1]{ReSiVi11}.

\begin{prop}\label{prop: general result to get binomial ideal}
Let $\theta\colon  K[y_1,\dots,y_s] \to K[x_1,\dots,x_n]$ be a ring homomorphism with 
$\theta(y_i)$ a monomial, for all $i=1,\dots,s$. Let $I\subseteq K[x_1,\dots,x_n]$ 
be an ideal ge\-ne\-ra\-ted by a finite number of homogeneous binomials. 
Then, $\theta^{-1}(I)$ is the intersection with $K[y_1,\dots,y_s]$ 
of the ideal of $K[x_1,\dots,x_n,z,y_1,\dots,y_s]$ generated by 
\begin{equation}\label{eq: L205}
\set{y_i-\theta(y_i)z : i=1,\dots,s}\cup I.
\end{equation}
Moreover, $\theta^{-1}(I)$ is an ideal generated by a finite number of homogeneous binomials.
\end{prop}

\begin{proof}
Let us denote by $J\subseteq K[x_1,\dots,x_n,z,y_1,\dots,y_s]$ the ideal generated by  \eqref{eq: L205}.
Since $\theta$ is a graded ring homomorphism, $\theta^{-1}(I)$ is a homogeneous ideal. Thus, to prove 
the inclusion $\theta^{-1}(I)\subseteq J\cap K[y_1,\dots,y_s]$, it suffices to restrict to homogeneous polynomials. 
Assume that $f\in K[y_1,\dots,y_s]$, homogeneous of degree $d$, is such that
$$
\ts\theta(f)=f(\theta(y_1),\dots,\theta(y_s))\in I.
$$ 
For each $i$ consider the substitution of $y_i$ in $f$ by $(y_i-\theta(y_i)z) + \theta(y_i)z$. Using the binomial theorem, we deduce that 
$$
\ts f = z^d f(\theta(y_1),\dots,\theta(y_n)) + \sum_{i=1}^s h_i \cdot (y_i-\theta(y_i)z),
$$
for some $h_i\in K[x_1,\dots,x_n,z,y_1,\dots,y_s]$. Since, by assumption, 
$\theta(f) \in I$ we conclude that $f \in J\cap K[y_1,\dots,y_s]$. 
\smallskip 

\noindent 
To prove the opposite inclusion, $J\cap K[y_1,\dots,y_s]\subseteq \theta^{-1}(I)$, we will 
show first that the ideal $J\cap K[y_1,\dots,y_s]$ is generated by binomials.
As, by assumption, $\theta(y_i)$ are monomials, every element of the set $\set{y_i-\theta(y_i)z : i=1,\dots,s}$ is a binomial.
This is also true for the given generating set of $I$. We deduce that $J$ is generated by a finite number of binomials. 
As $S(f,g)$, when $f$ and $g$ are binomials, if non-zero, is also a binomial, and the remainder of the division of a binomial 
by another binomial, if non-zero, is also a binomial, Buchberger's algorithm, 
for producing a Gr\"obner basis from the set of generators of $J$, will also yield a set 
of binomials. Using the elimination order (variables
$y_1,\dots,y_s$ as last variables) we deduce that $J\cap K[y_1,\dots,y_s]$ has a Gr\"obner basis consisting of binomials and thus, 
in particular, it is generated by binomials.
Accordingly, assume that $\y^\delta - \y^\gamma$, for some $\delta,\gamma \in \NN^s$ belongs to $J\cap K[y_1,\dots,y_s]$. Then, 
there exist $h_i,g_j \in K[x_1,\dots,x_n,z,y_1,\dots,y_s]$ and $\ell_j \in I$ such that 
\begin{equation}\label{eq: L234}
\ts \y^\delta - \y^\gamma = \sum_{i=1}^s h_i\cdot (y_i-\theta(y_i)z) + \sum_{j=1}^k g_j \ell_j,
\end{equation}
for some $k\geq 0$. 
Substituting above each $y_i$ by $\theta(y_i)\in K[x_1,\dots,x_n]$ and the variable $z$ by $1$ we deduce that 
$$
\ts \theta(\y^\delta - \y^\gamma) = \sum_{j=1}^k g'_j\ell_j,
$$
for some $g'_j\in K[x_1,\dots,x_n]$. This proves the inclusion $J\cap K[y_1,\dots,y_s]\subseteq \theta^{-1}(I)$.
\smallskip

\noindent
We have shown that $\theta^{-1}(I)=J\cap K[y_1,\dots,y_s]$. Hence, in particular, $\theta^{-1}(I)$ is generated
by a finite number of polynomials of the form $\y^\delta - \y^\gamma$. To see that each of these must be homogeneous, 
we go back to \eqref{eq: L234} and substitute all the variables $x_1,\dots,x_n$ by $1$. Then, since $I$ is generated by binomials
we get $\ell_j (1,\dots,1)= 0$. Moreover, $\theta(y_i)(1,\dots,1)=1$, as, by
assumption, $\theta(y_i)$ are monomials. We deduce:
$$
\ts \y^\delta - \y^\gamma = \sum_{i=1}^s h'_i\cdot (y_i-z),
$$
for some $h'_i\in K[z,y_1,\dots,y_s]$. Substituting in the above $y_i$ by $z$ we get: 
$$
z^{\delta_1+\cdots+\delta_s}-z^{\gamma_1+\cdots+\gamma_s} = 0,
$$
which implies that $\delta_1+\cdots+\delta_s=\gamma_1+\cdots+\gamma_s$, i.e., that $\y^\delta-\y^\gamma$ is homogeneous.
\end{proof}

\begin{cor} \label{cor: the ideal is binomial}
$I(X_G)$ is generated by homogeneous binomials.  
\end{cor}

\begin{proof} 
Apply Proposition~\ref{prop: general result to get binomial ideal}  with 
$K[y_1,\dots,y_s]=K[E_G]$, $K[x_1,\dots,x_n]=K[V_G]$, 
$\theta(t_{ij})=x_ix_j$ and $I=(x_i^2-x_j^2 : i,j\in V_G)$.
\end{proof}

\begin{rmk}
Since, as was shown in the proof of Proposition~\ref{prop: one-dimensional and Cohen-Macaulay}, 
any monomial is regular on $K[E_G]/I(X_G)$, from a generating set of $I(X_G)$ consisting of binomials we  
obtain one in which all binomials $\t^\alpha-\t^\beta$ satisfy $\gcd(\t^\alpha,\t^\beta)=1$.
\end{rmk}

\subsection{Binomials and subgraphs} In Section~\ref{sec: joins and ears}, we will
use the following characterization of homogeneous binomials in $I(X_G)$.

\begin{prop}\label{prop: even degrees}
Let $\t^\alpha - \t^\beta$ a homogeneous binomial with $\gcd(\t^\alpha,\t^\beta)=1$ and  
let $H$ be the subgraph of $G$ the edge set of which is in bijection with the 
variables that occur in either $\t^\alpha$ or $\t^\beta$ raised to an odd power. 
Then $\t^\alpha - \t^\beta \in I(X_G)$ if and only if the degree of $v$ in $H$ is even, 
for all $v \in V_H$. In particular, if $\set{i,k} \neq \set{j,\ell}$ are two edges then, 
$t_{ik}-t_{j\ell} \notin I(X_G)$.
\end{prop}

\begin{proof}
Let $\x^\delta -\x^\gamma\in K[V_G]$, with $\delta,\gamma \in \NN^{V_G}$ be a homogeneous binomial of degree $>1$. 
We claim that $\x^\delta-\x^\gamma \in (x_i^2-x_j^2 : i,j\in V_G)$ if, and only if, $\delta(i)+\gamma(i)$ is even, 
for every $i\in V_G$. 
\smallskip

\noindent 
To prove this claim, assume first that $\delta(i)+\gamma(i)$ is even, 
for every $i\in V_G$ and let us show that $\x^\delta +\x^\gamma\in (x_i^2-x_j^2 : i,j\in V_G)$. 
We will argue by induction on the degree of $\x^\delta - \x^\gamma$. 
Since $\x^\delta - \x^\gamma\not = 0$ there exists $i\in V_G$ such that $\delta(i)\not = \gamma(i)$.
If the degree of $\x^\delta - \x^\gamma$ is two and $\delta(i)+\gamma(i)$ is even, 
one of $\delta(i)$ or $\gamma(i)$ must be equal to $2$ and 
the other equal to $0$. Assume, without loss of generality that $\delta(i)=2$ and $\gamma(i)=0$.
Then, there exists $j\not = i$ such that $\delta(j)=0$ and $\gamma(j)=2$. In other words,
$\x^\delta-\x^\gamma = x_i^2-x_j^2$. Assume now that the degree of $\x^\delta-\x^\gamma$ is $>2$ and, 
without loss of generality, that $\delta(i)\geq \gamma(i)+2$. Let $j\in V_G$ be such that $\gamma(j)>0$
and let $\delta',\gamma' \in \NN^{V_G}$ be such that 
$\x^{\delta'} = \x^\delta / x_i^2$ and $\x^{\gamma'} = \x^\gamma/x_j$. Then 
\begin{equation}\label{eq: L466}
\x^\delta-\x^\gamma = (x_i^2-x_j^2)\x^{\delta'} +x_j(x_j\x^{\delta'}-\x^{\gamma'}).
\end{equation}
Write $x_j\x^{\delta'}-\x^{\gamma'} = \x^\mu-\x^\nu$ for some $\mu,\nu \in \NN^{V_G}$.
Then $\x^\mu+\x^\nu$ has degree one less than $\x^\delta+\x^\gamma$. Additionally
$$
\ts \mu(i)+\nu(i)=\delta(i)-2+\gamma(i),\quad  \mu(j)+\nu(j)= 1+\delta(j) +\gamma(j)-1
$$
are even, and so are $\mu(k)+\nu(k)=\delta(k)+\gamma(k)$, for every $k\in V_G\setminus \set{i,j}$.
By induction hypothesis, $\x^\mu-\x^\nu \in (x_i^2-x_j^2 : i,j \in V_G)$ and then, by \eqref{eq: L466},
$$
\x^\delta - \x^\gamma \in (x_i^2-x_j^2 : i,j \in V_G).
$$

\noindent
Conversely, let us assume that $\x^\delta -\x^\gamma \in (x_i^2-x_j^2 : i,j \in V_G)$.
We want to show that $\delta(i)+\gamma(i)$ is even, for every $i\in V_G$. Write 
$$
\ts \x^\delta - \x^\gamma = \sum_{ij} f_{ij} (x_i^2-x_j^2),
$$
for some $f_{ij}\in K[V_G]$ and, fixing $i\in V_G$, substitute in the above all
$x_j$ by $1$, for all $j\not = i$. Then, there exists $g\in K[x_i]$ such that 
$$
x_i^{\delta(i)} - x_i^{\gamma(i)} = g(x_i)(x_i^2-1).
$$
Without loss of generality we may assume that $\delta(i)>\gamma(i)$. Then we deduce that 
$$
g(x_i)=x_i^{\delta(i)-2}+x_i^{\delta(i)-4}+\cdots +x_i^{\delta(i)-2m},
$$
for some $m>0$, which, in particular, implies that $\gamma(i)=\delta(i)-2m$ and, therefore,
that $\delta(i)+\gamma(i)$ is even. We have proved our claim.
\smallskip

Let $\t^\alpha - \t^\beta$ be a homogeneous binomial with $\gcd(\t^\alpha,\t^\beta)=1$. 
Write 
$$
\theta(\t^\alpha-\t^\beta) = \x^\delta-\x^\gamma,
$$
for some $\delta,\gamma \in \NN^{V_G}$. 
Then, since $\gcd(\t^\alpha,\t^\beta)=1$, 
we deduce that $\delta(i)+\gamma(i)$
differs from $\deg_H (i)$ by an even number, for every $i\in V_G$. 
If $\x^\delta-\x^\gamma$ is zero then $\t^\alpha -\t^\beta \in I(X_G)$ and $\delta(i)=\gamma(i)$, for every $i\in V_G$, which
implies that $\delta(i)+\gamma(i)$ is even. Assume that $\x^\delta-\x^\gamma$ is non-zero, 
and, thus, a homogeneous binomial of degree $\geq 2$. Then,  $\t^\alpha -\t^\beta \in I(X_G)$ if
and only if, by definition, \mbox{$\x^\delta-\x^\gamma \in (x_i^2-x_j^2 : i,j\in V_G)$} which,
by our claim and previous observation, is equivalent to $\deg_H(i)$ being even, for 
every $i\in V_G$.
\end{proof}

\begin{examp}
Consider the graph, $G$, in Figure~\ref{fig: the non-bipartite graph}.
Then, by Proposition~\ref{prop: even degrees},
$$
\renewcommand{\arraystretch}{1.5}
\begin{array}{l}
t_{13}t_{45}t_{56}-t_{12}t_{23}t_{46}, \quad t_{23}t_{45}t_{56}-t_{12}t_{13}t_{46},\quad t_{12}t_{45}t_{56}-t_{23}t_{13}t_{46},\\ 
t_{23}t_{13}t_{56}-t_{12}t_{45}t_{46}, \quad t_{12}t_{13}t_{56}-t_{23}t_{45}t_{46}, \quad t_{12}t_{23}t_{56}-t_{13}t_{45}t_{46},\\
t_{23}t_{13}t_{45}-t_{12}t_{56}t_{46}, \quad t_{12}t_{13}t_{45}-t_{23}t_{56}t_{46}, \quad t_{12}t_{23}t_{45}-t_{13}t_{56}t_{46},\\ t_{12}t_{23}t_{13}-t_{45}t_{56}t_{46}.
\end{array}
$$ 
are binomials belonging to $I(X_G)$, as they are all associated to the subgraph given by the two triangles
of $G$. 
\begin{figure}[ht]
\begin{center}
\begin{tikzpicture}[line cap=round,line join=round, scale=1.5]
\draw [fill=black] (-.5,0) circle (1pt) node[anchor = south west] {$\ss 3$};
\draw [fill=black] (.5,0) circle (1pt) node[anchor = south east] {$\ss 4$};
\draw [fill=black] (-1.25,.5) circle (1pt) node[anchor = south east] {$\ss 1$};
\draw [fill=black] (-1.25,-.5) circle (1pt) node[anchor = north east] {$\ss 2$};
\draw [fill=black] (1.25,.5) circle (1pt) node[anchor = south west] {$\ss 5$};
\draw [fill=black] (1.25,-.5) circle (1pt) node[anchor = north west] {$\ss 6$};
\draw (-1.25,.5)-- (-1.25,-.5);
\draw (1.25,.5) -- (1.25,-.5);
\draw (-1.25,.5) --(-.5,0);
\draw (-1.25,-.5) --(-.5,0);
\draw (1.25,.5) --(.5,0);
\draw (1.25,-.5) --(.5,0);
\draw (-.5,0) -- (.5,0);
\end{tikzpicture}
\end{center}
\caption{A non-bipartite graph.}
\label{fig: the non-bipartite graph} 
\end{figure}
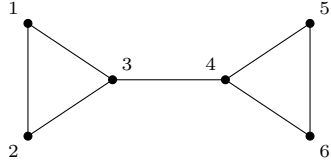
Using \cite[Macaulay2]{M2} one can show that these binomials together with 
$\set{t^2_{ij}-t_{46}^2 :\; \ss \set{i,j}\ts \in E_G\setminus \set{\ss\set{4,6}\ts}}$
give a reduced Gr\"obner basis for $I(X_G)$, with respect to the graded reverse lexicographic order 
induced by $t_{12}\succ t_{23} \succ t_{13}\succ t_{34} \succ t_{45} \succ t_{56} \succ t_{46}$.

\end{examp}

\begin{cor}\label{cor: relation of the ideal with the ideal with respect to a subgraph}
Let $H\subseteq G$ be a subgraph without isolated vertices. Consider $I(X_H)$ 
as a subset of $K[E_G]$ under the inclusion $K[E_H]\subseteq K[E_G]$. Then 
$$I(X_H) = I(X_G) \cap K[E_H].$$
\end{cor}

\begin{proof}
Since both $I(X_H)$ and $I(X_G)$ are generated by homogeneous binomials, we may restrict to checking that 
$I(X_H)$ and $I(X_G)\cap K[E_H]$ contain the same homogeneous binomials $\t^\alpha-\t^\beta \in K[E_H]$, with
$\gcd(\t^\alpha,\t^\beta)=1$. This follows from Proposition~\ref{prop: even degrees}.
\end{proof}

\subsection{Independence of the field} The construction of Proposition~\ref{prop: general result to get binomial ideal} 
can be used to show that, for a fixed monomial order, there exists a set of generators of the initial 
ideal of $I(X_G)$ which is independent of the field.

\begin{prop}\label{prop: initial ideal is independent of the field}
For a fixed monomial order, there exists a set of generators of the initial ideal of $I(X_G)$ which is independent of the field.
Moreover, if the characteristic of the field is not $2$, then there exists a Gr\"obner basis of 
$I(X_G)$ independent of the field.
\end{prop}

\begin{proof}
Consider a monomial order in $K[E_G]$.
Let us apply Proposition~\ref{prop: general result to get binomial ideal}  with 
$$K[y_1,\dots,y_s]=K[E_G],\quad K[x_1,\dots,x_n]=K[V_G],$$ 
$\theta(t_{ij})=x_ix_j$ and $I=(x_i^2-x_j^2 : i,j\in V_G)$.
Fix a monomial order in 
$$
K[\set{x_i : i\in V_G}\cup \set{z}\cup \set{t_e : e\in E_G}]
$$
extending the monomial order of $K[E_G]$ in which the variables in $\set{t_e:e\in E_G}$, are 
the least variables. Then, by the argument in the proof of Proposition~\ref{prop: general result to get binomial ideal},
a Gr\"obner basis of $I(X_G)$ can be obtained by first applying  Buchberger's algorithm to the set
$$
\ts \set{t_e - \theta(t_e)z : e\in E_G} \cup \set{x_i^2-x_j^2 : i,j \in V_G},
$$
and then taking only the elements of the output belonging $K[E_G]$. 
Assume that the field has characteristic $\not = 2$. Then, 
since the \mbox{$S$-polynomial} of the difference of two monomials, if it is not zero, 
is again a difference of two monomials, since 
the remainder in the standard expression of a binomial with respect to a list of binomials, if it is not zero,
is again a binomial, the result of Buchberger's algorithm will be independent of the field.
If the characteristic of $K$ is $2$, in which case one replaces ``difference of monomials'' by ``sum of monomials''
in the previous argument, the Gr\"obner basis obtained yields a set of leading terms with coefficient $1$. 
The ones obtained in the previous case may differ, possibly, by the multiplication with $-1$.
\end{proof}

\subsection{Relation to the vanishing ideal over graphs}\label{subsec: where you introduce the vanishing ideal}
Let us now recall from \cite{ReSiVi11} the notion of vanishing ideal over a graph. 
Assume that $K$ is a finite field, let 
$\PP^{|V_G|-1}$ and $\PP^{|E_G|-1}$ 
be the projective spaces with coordinates indexed 
by $V_G$ and $E_G$, respectively and let 
$$
\vartheta\colon \PP^{|V_G|-1}\dasharrow \PP^{|E_G|-1}
$$ 
be the rational map defined by $t_e \mapsto x_ix_j$, for every edge $e=\set{i,j}$.
The \emph{projective toric subset parameterized by $G$} is the subset of $\PP^{|E_G|-1}$ given by 
the image by $\vartheta$ of the projective torus, 
$$
\ts \TT^{|V_G|-1} = \set{(x_i)_{i\in V_G} \in \PP^{|V_G|-1} : \prod_{i\in V_G} x_i \not = 0} \subseteq \PP^{|V_G|-1}.
$$ 
The vanishing ideal over $G$, for the finite field $K$, is, by definition, the vanishing ideal 
of this set. These ideals were defined and studied in \cite{ReSiVi11} and, since then, appeared in the literature in 
\cite{GoRe08, GoReSa13, MaNeVPVi, Ne, NeVP14,NeVPVi15, NeVPVi14, SaVPVi11, VPVi13}. 
We know that the vanishing ideal over a graph is a binomial ideal and that the quotient 
of $K[E_G]$ by it is a Cohen--Macaulay, reduced, one-dimensional graded ring. However the invariants
of the vanishing ideal and its minimal generating sets are not independent of the
field. For example, if $G$ is a cycle of length $4$ on the vertex set $V_G=\set{1,2,3,4}$, with 
$K$ a field with $q$ elements, then, by \cite[Theorem~5.9]{NeVPVi15}, we know that the vanishing ideal over $G$ is
generated by:
$$
\renewcommand{\arraystretch}{1.75}
\begin{array}{l}
t_{12}t_{34}-t_{23}t_{14}, \\
t^{q-1}_{12}-t_{14}^{q-1},\; t^{q-1}_{23}-t_{14}^{q-1},\; t^{q-1}_{34}-t_{14}^{q-1},\\
t_{12}^{q-2}t_{23}-t_{34}t_{14}^{q-2},\; t_{12}^{q-3}t_{23}^2-t_{34}^2t_{14}^{q-3},\dots,\; t_{12}t_{23}^{q-2}-t_{34}^{q-2}t_{14}\\
t_{12}^{q-2}t_{14}-t_{34}t_{23}^{q-2},\; t_{12}^{q-3}t_{14}^2-t_{34}^2t_{23}^{q-3},\dots,\; t_{12}t_{14}^{q-2}-t_{34}^{q-2}t_{23}.\\
\end{array}
$$
By \cite[Theorems~3.2~and~6.2]{NeVPVi15}, 
this ideal has degree $(q-1)^{2}$ and regularity $q-1$.
\smallskip 

\noindent
The link between the ideals $I(X_G)$ and the vanishing ideals occurs when $K=\ZZ_3$. As we show below,
in this situation, both ideals coincide.

\begin{prop}\label{prop: Link to vanishing ideals}
Assume that $K=\ZZ_3$. 
Then $I(X_G)$ is the vanishing ideal of the projective toric subset parameterized by $G$
\end{prop}

\begin{proof}
Since $I(X_G)$ is a homogeneous ideal, it will suffice to check 
that $I(X_G)$ and the vanishing ideal of the projective toric subset 
parameterized by $G$ have the same homogeneous elements. Let $f\in K[E_G]$ be homogeneous. 
Then $f$ vanishes on the projective toric subset parameterized by $G$ if and only if $\theta(f)$ vanishes on 
$\TT^{|V_G|-1}$. Now, if $K=\ZZ_3$, the vanishing ideal of $\TT^{|V_G|-1}$ is 
$$
(x_i^2-x_j^2 : i,j\in V_G).
$$ 
We deduce that $\theta(f)$ vanishes on $\TT^{|V_G|-1}$ if and only if $\theta(f) \in (x_i^2-x_j^2 : i,j\in V_G)$ which, by
definition, is equivalent to $f\in I(X_G)$.
\end{proof}

\begin{rmk}\label{rmk: the set in the notation}
One can define a subset, $X_G\subseteq \PP^{|E_G|-1}$ over any field $K$, by the image of 
a map $\vartheta$, defined exactly as above, of a subset $D\subseteq \PP^{|V_G|-1}$ consisting 
of all points in $\PP^{|V_G|-1}$ with homogeneous coordinates in 
$\set{-1,1}\subseteq K$. One can then show that, if the characteristic of $K$ is not equal to $2$, 
$I(X_G)$ is the vanishing ideal of $X_G$. 
Note, however, that this is definitely not the case if the characteristic of $K$ is equal to $2$, 
as then $D$ consists of a single point. 
\end{rmk}

\subsection{Degree of the ideal} 
The next result relates the degree of $I(X_G)$ with $|V_G|$ and $b_0(G)$, 
the number of connected components of $G$. 
Recall that we are assuming that $G$ is a simple graph without isolated vertices.

\begin{prop}\label{prop: degree of the ideal}
The degree of $I(X_G)$ is $2^{|V_G|-b_0(G)}$, if $G$ is non-bipartite, and it is $2^{|V_G|-b_0(G)-1}$ if $G$ is bipartite.
\end{prop}

\begin{proof}
Fix a monomial order in $K[E_G]$. 
By Proposition~\ref{prop: initial ideal is independent of the field} there exists a generating 
set for the initial ideal of $I(X_G)$ independent of the field. Hence  
the multiplicity degree of $K[E_G]/I(X_G)$ is independent of the field. 
Consider then $K=\ZZ_3$. By Proposition~\ref{prop: Link to vanishing ideals}, 
$I(X_G)$ is the vanishing ideal over the graph $G$. The
result now follows from \cite[Theorem~3.2]{NeVPVi15}.
\end{proof}

\section{Regularity}\label{sec: regularity}
The Castelnuovo--Mumford regularity of a graded finitely generated module, $M$, over a polynomial ring
is, by definition, 
$$
\ts \reg M=\max_{i,j} \set{j-i : \beta_{ij} \not = 0},
$$
where $\beta_{ij}$ are the graded Betti numbers of $M$. The index of regularity of $M$ is, by definition, 
$$
\operatorname{ir} M =\min \set{k\in \mathbb{N} : \varphi_M(d)=P_M(d),\; \forall\, d\geq k},
$$
where $\varphi_M$ and $P_M$ are the Hilbert function and the Hilbert Polynomial of $M$, respectively.
By \cite[Corollary~4.8]{E05}, if $M$ is Cohen--Macaulay,  
$$
\reg M = \operatorname{ir} M +\dim M -1,
$$
and hence the Castelnuovo--Mumford regularity and the index of regularity of the module $K[E_G]/I(X_G)$ coincide.
Since 
$$
\reg I(X_G) = \reg K[E_G]/I(X_G) + 1,
$$
and since the Hilbert function of $K[E_G]/I(X_G)$ and the Hilbert function 
of the quotient of $K[E_G]$ by the initial ideal of $I(X_G)$ coincide, using 
Proposition~\ref{prop: initial ideal is independent of the field}, the proof of the following result is straightforward. 

\begin{prop}\label{prop: regularity is independent of the field}
The regularity of $I(X_G)$  is independent of the field.
\end{prop}

Together with Proposition~\ref{prop: Link to vanishing ideals}, this result enables 
the transfer to our setting of some results about the regularity of the vanishing ideal of a graph.

\begin{prop}\label{prop: known properties and values}
Let $G$ be a simple graph without isolated vertices.
\begin{enumerate}\setlength{\itemsep}{.25cm}
\item If $H\subseteq G$ is a spanning subgraph with the same number of connected components such that \mbox{either}  
$H$ and $G$ are both non-bipartite or both bipartite, then
$\reg I(X_G)\leq \reg I(X_H)$.
\item If $G$ is bipartite and $H_1,\dots,H_c$ are its blocks, then
$$
\ts \reg I(X_G) =  \sum_{i=1}^c \reg I(X_{H_i}). 
$$
\item If $G=\K_{a,b}$ is a complete bipartite graph, then $\reg I(X_G) = \max\set{a,b}$. 
\item If $G$ is bipartite and Hamiltonian, then $\reg I(X_G) =\frac{|V_G|}{2}$.
\item If $G$ is a forest, then $\reg I(X_G) =|V_G|-b_0(G)$.
\item If $G$ has a single cycle and this cycle is odd, then $$\reg I(X_G) = |V_G|-b_0(G)+1.$$
\end{enumerate}
\end{prop}

\begin{proof}
By Proposition~\ref{prop: regularity is independent of the field}, we may consider $K=\ZZ_3$, in which case,  
by Proposition~\ref{prop: Link to vanishing ideals}, $I(X_G)$ is the vanishing ideal of $X_G$, the projective toric subset
parameterized by $G$.
\smallskip

\noindent
(i) Since $V_H=V_G$ and either $H$ and $G$ are
both non-bipartite or both bipartite, we get, 
from Proposition~\ref{prop: degree of the ideal}, $|X_H|=|X_G|$. 
Using \cite[Lemma~2.13]{VPVi13},
$$ 
\ts \reg I(X_G) = K[E_G]/I(X_G)+1\leq \reg K[E_H]/I(X_H)+1 = \reg I(X_H).
$$

\noindent
(ii) Using \cite[Theorem~7.4]{NeVPVi14}, 
$$
\ts \reg I(X_G) =  \sum_{i=1}^c \reg K[E_{H_i}]/I(X_{H_i}) + (c-1) +1 = \sum_{i=1}^c \reg I(X_{H_i}). 
$$

\noindent
(iii) Using \cite[Corollary~5.1.9]{monalg} and \cite[Corollary~5.4]{GoRe08}, 
$$
\reg I(X_G) = \reg K[E_G]/I(X_G) +1 = \max\set{a,b}-1+1 = \max\set{a,b}.
$$ 

\noindent
(iv) If $G$ is an even cycle then this follows from \cite[Theorem~6.2]{NeVPVi15}.
Consider the general case. Then $G$ is connected and its vertex set has even cardinality. Let $a$ be
equal to $\frac{|V_G|}{2}$ and let $C$ denote an (even) Hamiltonian cycle. Then $C\subseteq G\subseteq \K_{a,a}$.
By (i), (iii) and the even cycle case described before:
$$
\ts a \leq \reg I(X_G) \leq \frac{|V_G|}{2} = a. 
$$

\noindent
(v) Suppose $G$ is an edge. Then $I(X_G)=(0)$ and hence \mbox{$\reg I(X_G)=1=|V_G|-1$}. 
If $G$ is a forest, then $G$ is bipartite and each edge is a block of $G$, hence, by (ii), we get
$\reg I(X_G) = |E_G|=|V_G|-b_0(G)$. 
\smallskip

\noindent
(vi) If $G$ has a single odd cycle then, 
from Proposition~\ref{prop: degree of the ideal}, we get
$$
\ts |X_G|=2^{|V_G|-b_0(G)}=2^{|E_G|-1}.
$$ 
We deduce that $X_G$ coincides with the projective torus, $\TT^{|E_G|-1}\subseteq \PP^{|E_G|-1}$ and,
accordingly, $I(X_G) = (t_e^2-t_f^2 : e,f\in E_G)$,
which is a complete intersection of $|E_G|-1$ binomials of degree two.
The Hilbert series of the quotient  
$K[E_G]/I(X_G)$ is then equal to 
$$
\ts \frac{(1-T^2)^{|E_G|-1}}{(1-T)^{|E_G|}}\cdot
$$
By \cite[Corollary~5.1.9]{monalg},
$\reg I(X_G) =  2|E_G|-2-|E_G|+2=|V_G|-b_0(G)+1$.
\end{proof}

\begin{rmk}
If $G$ is bipartite, it is straightforward from Proposition~\ref{prop: known properties and values} that  
the regularity of $I(X_G)$ is additive on the connected components of $G$. 
This does not hold without the bipartite assumption. In fact, even if $G$ is connected, 
addi\-tivity along the blocks of $G$ does not hold without the bipartite assumption. 
A counter-example is given by the graph in Figure~\ref{fig: the non-bipartite graph}. 
The three blocks of the graph yield regularities $3$ (twice) and $1$. However,  
using \cite[Macaulay2]{M2} one checks that the regularity of $I(X_G)$ is $4$.
\end{rmk}

The next results reflect our approach to the computation of the regularity of 
$I(X_G)$ or, equivalently, the regularity of the quotient $K[E_G]/I(X_G)$. We will 
resort to an Artin\-ian reduction of $K[E_G]/I(X_G)$, by quotienting 
the polynomial ring by the ideal generated by $I(X_G)$ and an arbitrary 
monomial $\t^\delta$ of degree $d$. As we saw in the proof of Proposition~\ref{prop: one-dimensional and Cohen-Macaulay}, 
$\t^\delta$ is $K[E_G]/I(X_G)$-regular and therefore multiplication by $\t^\alpha$ yields  
the following short exact sequence:
\begin{equation}\label{eq: L478}
0 \to \frac{K[E_G]}{I(X_G)} [-d] \stackrel{\t^\delta}{\longrightarrow} \frac{K[E_G]}{I(X_G)} \to \frac{K[E_G]}{(I(X_G),\t^\alpha)} \to 0.
\end{equation} 

\begin{prop}\label{prop: computing reg by reducing to Artinian quotient}
Let $\t^\delta \in K[E_G]$ be a monomial of degree $d$. Then the quotient 
$K[E_G]/(I(X_G),\t^\delta)$ is zero in degree $k$ if and only if $k\geq \reg K[E_G]/I(X_G) + d$.
\end{prop}

\begin{proof}
Let $\varphi$ denote the Hilbert function of $K[E_G]/I(X_G)$.
Then the quotient $K[E_G]/(I(X_G),\t^\alpha)$ is zero in degree $k$ if and only if $K[E_G]/(I(X_G),\t^\alpha)$ is zero in 
every degree $i\geq k$. By \eqref{eq: L478}, this is equivalent to $\varphi(i-d)=\varphi(i)$, for every 
$i\geq k$, which holds if and only if $\varphi$ is constant from $k-d$ and on, i.e., if and only if, $\reg K[E_G]/I(X_G)\leq k-d$. 
\end{proof}

\begin{prop}\label{prop: special replacement}
Let $\t^\delta \in K[E_G]$ be a monomial. Then,
$\t^\alpha \in (I(X_G),\t^\delta)$ if and only if 
there exists a monomial $\t^\beta \in K[E_G]$
such that $\t^\alpha - \t^\beta$ is homogeneous, belongs to $I(X_G)$, and
$\t^\delta \mid \t^\beta$.
\end{prop}

\begin{proof}
Fix a monomial order on $K[E_G]$. Let $\G$ be a Gr\"obner basis of $I(X_G)$ obtained as in
the proofs of Proposition~\ref{prop: general result to get binomial ideal} or 
Proposition~\ref{prop: initial ideal is independent of the field}. Then each element 
of $\G$ is a homogeneous binomial, $\t^\alpha-\t^\beta$. Furthermore, we may assume, 
without loss of generality, that $\operatorname{lt}(\t^\alpha- \t^\beta)=\t^\alpha$.
\smallskip 

\noindent
We claim that $(I(X_G),\t^\delta)$
has a Gr\"obner basis of the form $\G\cup \set{\t^{\mu_1},\dots,\t^{\mu_r}}$, where,
for each $i=1,\dots,r$, there exists $\t^{\beta_i} \in K[E_G]$ such that $\t^{\mu_i}-\t^{\beta_i}$ is homogeneous, 
belongs to $I(X_G)$ and $\t^\delta \mid \t^{\beta_i}$.
To prove this claim it suffices to show that there exists an application of Buchberger's algorithm 
which, starting from $\G\cup \set{\t^\delta}$, produces in step $i$ a set
$$
\G_i = \G \cup \set{\t^{\mu_1},\dots,\t^{\mu_{i+1}}}
$$
with the stated properties. We prove this by induction on $i\geq 0$. If $i=0$, the algorithm is in the initialization 
step and hence $\G_0 = \G \cup \set{\t^\delta}$. It suffices to set $\mu_1=\beta_1 = \delta$. Now fix $i>0$ and 
assume Buchberger's algorithm has not finished in the step $i-1$. Then there is an $S$-polynomial which does not reduce to zero 
modulo $\G_{i-1}=\G \cup \set{\t^{\mu_1},\dots,\t^{\mu_{i}}}$. Since the $S$-polynomial of two monomials is zero and the $S$-polynomial 
of two elements of $\G$ reduces to zero modulo $\G$, the $S$-polynomial in question must be 
of some $\t^{\mu_k}\in \set{\t^{\mu_1},\dots,\t^{\mu_{i}}}$ and $\t^\alpha-\t^\beta \in \G$. Let us
write
\begin{equation}\label{eq: L87}
\ts S(\t^{\mu_k},\t^\alpha-\t^\beta) = \t^{\mu'}\t^{\mu_k} - \t^{\alpha'}(\t^\alpha -\t^\beta)  = \t^{\alpha'}\t^\beta,
\end{equation}
where $\t^{\mu'}$ and $\t^{\alpha'}$ are such that $\t^{\mu'}\t^{\mu_k} = \t^{\alpha'}\t^\alpha = \operatorname{lcm}(\t^{\mu_k},\t^\alpha)$.
Let $r$ be the remainder of $S(\t^{\mu_k},\t^{\alpha}-\t^{\beta})$ in its standard expression with respect to $\G_{i-1}$. Since 
$S(\t^{\mu_k},\t^{\alpha}-\t^{\beta})$ is a monomial and $r\not = 0$, to obtain $r$ only division by the elements of $\G$ is carried out.
Since division of a monomial by a binomial yields a monomial of the same degree, 
we deduce that $r=\t^{\mu_{i+1}}$, for some monomial $\t^{\mu_{i+1}} \in K[E_G]$ with $\deg(\t^{\mu_{i+1}})= \deg(\t^{\alpha'}\t^\beta)$, and 
that there exists $g\in I(X_G)$ such that
$$
S(\t^{\mu_k},\t^{\alpha}-\t^{\beta}) = g + r \iff  \t^{\alpha'}\t^\beta= g + \t^{\mu_{i+1}} \iff \t^{\mu_{i+1}}-\t^{\alpha'}\t^\beta \in I(X_G).
$$
Using \eqref{eq: L87} and the induction hypothesis,
$$
\t^{\mu_{i+1}} - \t^{\mu'}\t^{\mu_k} \in I(X_G) \iff \t^{\mu_{i+1}} - \t^{\mu'}\t^{\beta_k} \in I(X_G), 
$$
where $\deg(\t^{\mu_k})=\deg(\t^{\beta_k})$. As 
$$
\ts \deg(\t^{\mu_{i+1}})=\deg(\t^{\alpha'}\t^\beta)=\deg(\t^{\alpha'}\t^\alpha) =\deg(\t^{\mu'}\t^{\mu_k}) = \deg(\t^{\mu'}\t^{\beta_k}),
$$
if we set $\beta_{i+1}=\mu'+\beta_k$, we see that $\G_i=\G \cup \set{\t^{\mu_1},\dots,\t^{\mu_{i+1}}}$, 
obtained in this step, satisfies the properties of the claim.
\smallskip

\noindent
Let us now use the Gr\"obner basis $\G \cup \set{\t^{\mu_1},\dots, \t^{\mu_r}}$ of the ideal $(I(X_G),\t^\delta)$
to prove this proposition. Let $\t^\alpha \in K[E_G]$ belong to this ideal. Then the remainder in its standard expression 
with respect to  $\G \cup \set{\t^{\mu_1},\dots, \t^{\mu_r}}$ is zero. As the remainder of the division of $\t^\alpha$ 
by a binomial is a monomial of the same degree, the division of $\t^\alpha$ by the elements of 
the Gr\"obner basis finishes the first time we use an element of 
the set $\set{\t^{\mu_1},\dots, \t^{\mu_r}}$. This means that there exists
$k\in \set{1,\dots,r}$, $\t^{\alpha'} \in K[E_G]$ and $g\in I(X_G)$ such that 
\begin{equation}\label{eq: L112}
\t^\alpha = g + \t^{\alpha'} \t^{\mu_k} \iff \t^\alpha - \t^{\alpha'} \t^{\mu_k} \in I(X_G),
\end{equation}
with $ \t^\alpha - \t^{\alpha'} \t^{\mu_k}$ homogeneous.
Let $\t^{\beta_k}\in K[E_G]$ be such that $\t^\delta \mid \t^{\beta_k}$ and such that $\t^{\mu_k}-\t^{\beta_k}$ is homogeneous and belongs to $I(X_G)$. Then, 
setting $\t^\beta = \t^{\alpha'}\t^{\beta_k}$, we see that $\t^\alpha-\t^\beta$ is homogeneous, by \eqref{eq: L112} that
it belongs to $\t^{\alpha} - \t^\beta \in I(X_G)$, and that $\t^\delta \mid \t^\beta$. 
We have proved one implication in the statement of the proposition, the other is trivial.
\end{proof}

\section{Joins and ears of graphs}\label{sec: joins and ears}

\subsection{Regularity and the maximum vertex join number} 
We now derive the bounds for the regularity of $I(X_G)$ mentioned in the introduction of this 
article. The lower bound, which
is the maximum vertex join number of the graph, gives the value of $\reg I(X_G)$ in the bipartite case. 
The proofs of this section rely on 
Propositions~\ref{prop: computing reg by reducing to Artinian quotient} and \ref{prop: special replacement}.

\begin{definition}{\cite{Fr93,SoZa93}}
A \emph{join} of a graph, $G$, is a set of edges, $J\subseteq E_G$, such 
that, for every circuit $C$ in $G$, $|J\cap E_C|\leq |E_C|/2$.
\end{definition}

Recall from Definition~\ref{def: max vertex join number}, that the maximum cardinality of a join of $G$ is called
the maximum vertex join number and is denoted by $\mu(G)$.

\begin{theorem}\label{thm: reg bigger than join} $\reg I(X_G)\geq \mu(G)$.
\end{theorem}

\begin{proof}
Let $J$ be a join of $G$ and let us show that  
$$
\reg K[E_G]/I(X_G) \geq |J|-1.
$$ 
Fix $e \in J$. 
By Proposition~\ref{prop: computing reg by reducing to Artinian quotient} it suffices to show that there exists a monomial of degree 
$|J|-1$ which does not belong to $(I(X_G),t_e)$. 
We will show that the product of variables corresponding to the edges in 
any subset of $J$ that does not include $e$
satisfies this property. We argue by induction.
Starting with the base case, 
let $f \in J \setminus \set{e}$. Then, by Proposition~\ref{prop: special replacement}, 
$$
t_f \in (I(X),t_e) \iff t_f-t_e \in I(X),
$$
which, by Proposition~\ref{prop: even degrees} is not true.

\smallskip
\noindent 
Assume now that $J' \subseteq J \setminus \set{e}$ is a subset of $k$ edges, with $k \geq 2$, and let $\t^\alpha$ be the product of all 
variables corresponding to edges of $J'$. We want to show that 
$\t^\alpha \notin (I(X_G),t_e)$. By the induction hypothesis, if $\t^\gamma$ is 
the product of variables corresponding to $k-1$ or fewer edges of $J \setminus \set{e}$ then $\t^\gamma \notin (I(X_G),t_e)$. 
We argue by contradiction. Suppose that $\t^\alpha \in (I(X_G),t_e)$. By Proposition~\ref{prop: special replacement}, there exists a monomial $\t^\beta$
such that $\t^\alpha - \t^\beta$ is homogeneous, $\t^\alpha - \t^\beta \in I(X_G)$, and
$t_e \mid \t^\beta$. Let $\t^\gamma$, $\t^\mu$ be such that $\t^\alpha=\t^\gamma\gcd(\t^\alpha,\t^\beta)$ and
$\t^\beta=\t^\mu\gcd(\t^\alpha,\t^\beta)$. Since any monomial is a regular element of $K[E_G]/I(X_G)$, 
we get $\t^\gamma - \t^\mu \in I(X_G)$. Since 
$t_e$ still divides $\t^\mu$, we deduce that 
$$
\t^\gamma \in (I(X_G),t_e).
$$ 
But $\t^\gamma$ cannot be the product of fewer than $k$ of $J \setminus \set{e}$, for otherwise  
we would have a contradiction with our induction hypothesis. Therefore, $\operatorname{gcd}(\t^\alpha,\t^\beta)=1$.

\smallskip
\noindent 
Let $H$ be the subgraph of $G$ the edges of which 
correspond to variables occurring in $\t^\alpha$ or $\t^\beta$ raised to an odd power. Notice that 
$J'\subseteq E_H$. Then, by Proposition~\ref{prop: even degrees}, every vertex of $H$ has even degree. 
By \cite[Theorem~1]{Bo98}, we conclude that $H$ decomposes into a union of edge disjoint cycles. 
Let $C_l\subseteq H\subseteq G$, for $l=1,\dots,r$, be the cycles satisfying $E_H=\sqcup_{l} E_{C_l}$.
Since $J' \subseteq E_H$ and $J'$ is a join, we get:
\begin{equation}\label{eq: L728}
\ts \mbox{deg}(\t^\alpha) = |J'| = \sum_l  |J' \cap E_{C_l}| \leq \frac{1}{2} \sum_l |E_{C_l}|
= \frac{1}{2} |E_H| \,.
\end{equation}
But, as $\t^\alpha - \t^\beta$ is homogeneous, we know that $|E_H|\leq 2\deg(\t^\alpha)$. 
By \eqref{eq: L728} this implies that  $|E_H|= 2 \deg(\t^\alpha)$ from which we deduce 
that all variables in $\t^\beta$ occur raised to $1$ and that,
therefore, $E_H$ contains all edges corresponding to variables dividing $\t^\beta$. In particular, $e\in E_H$.
Considering now the join $J' \cup \set{e}\subseteq J$, we get
$$
\ts \deg(\t^\alpha) +1 = |J' \cup \{e\}| = \sum_l |(J' \cup \{e\}) \cap E_{C_l}| \leq \frac{1}{2} \sum_l |E_{C_l}| = \frac{1}{2} |E_H| = \deg(\t^\alpha),
$$
which is a contradiction. We conclude that $\t^\alpha \notin (I(X_G),t_e)$, and, thus, finish the proof of the induction step.
\end{proof}

\begin{prop}\label{prop: general upper bound for reg}
\begin{equation}\label{eq: upper bounds}
\reg I(X_G) \leq \begin{cases} 
|V_G|-b_0(G), & \text{if $G$ is bipartite} \\
|V_G|-b_0(G)+1, & \text{if $G$ is non-bipartite.}
\end{cases}
\end{equation}
\end{prop}
\begin{proof}
Suppose $G$ is bipartite. Let
$H$ be a subgraph of $G$ consisting of a spanning tree for each connected component of $G$. Then, by 
Proposition~\ref{prop: known properties and values}, 
$$
\reg I(X_G)\leq \reg I(X_H) = |V_G|-b_0(G).
$$
Suppose now that $G$ is non-bipartite and take $H$, as before, given by a spanning tree for every connected 
component of $G$, except for one of the non-bipartite components in which we take a spanning connected graph containing a single
odd cycle. Then $H$ and $G$ have the same number of connected components, they are both non-bipartite and $H$ is a graph with a single 
odd cycle. According to Proposition~\ref{prop: known properties and values},
$$
\reg I(X_G) \leq \reg I(X_H) = |V_H|-b_0(H)-1 = |V_G|-b_0(G)-1. \qedhere
$$
\end{proof}

\begin{rmk}
The graphs $H$ in the proof of Proposition~\ref{prop: general upper bound for reg}, 
a forest in the bipartite case and a graph with a unique odd cycle in the non-bipartite
case, are examples of graphs for which the upper bounds \eqref{eq: upper bounds} are attained.  
From this observation and the next theorem, we deduce that the bounds for $\reg I(X_G)$ are sharp.
\end{rmk}

\begin{theorem}\label{thm: reg for bipartites is max join}
If $G$ is a bipartite graph, then $\reg I(X_G) = \mu(G)$.
\end{theorem}

\begin{proof}
We must show that 
\begin{equation}\label{eq: L923}
\reg K[E_G]/I(X_G) \leq \mu(G) - 1.
\end{equation}
Fix $t_e \in E_G$. According to Proposition~\ref{prop: computing reg by reducing to Artinian quotient},
to prove \eqref{eq: L923} it suffices to show that any monomial $\t^\alpha \in K[E_G]$ of degree $\mu(G)$ 
belongs to $(I(X_G), t_e)$. 
If $t_e \mid \t^\alpha$, we are done. Assume that $t_e\nmid \t^\alpha$. 
Suppose now that $t_f^2 \mid  \t^\alpha$, for some $f \neq e$. Then, 
setting $\t^\gamma = \t^\alpha / t^2_{f}$, 
$$
\t^\alpha = t_{f}^2\t^\gamma - t_{e}^2\t^\gamma + t_{e}^2\t^\gamma
= (t_{f}^2-t_{e}^2)\t^\gamma + t_{e}^2\t^\gamma \in (I(X_G),t_{e}).
$$
Hence we may assume that $t_{e}\nmid \t^\alpha$ and that $t_{f}^2\nmid  \t^\alpha$
for all $f\in E_G$. Let us denote by $H$ the subgraph of $G$ on the set of edges corresponding to the variables occurring in 
the monomial  $\t^\alpha$. Then
$$
|\,E_H \cup \{e\}\,| = \mu(G)+1
$$
and hence $E_H \cup \{e\}$ is not a join of $G$. I.e., there exists a circuit, $C$, in $G$ 
such that
\begin{equation}\label{eq: L957}
\ts |\,(E_H \cup \{e\})\, \cap \,E_C \,| > \,\frac{|E_C|}{2} \,.
\end{equation}
Since $G$ is bipartite, $C$ decomposes into an edge disjoint union of even cycles. Without loss
of generality, we may then assume that $C$ is an even cycle.
We will consider two cases.
In the first case, $e \in E_C$. Then 
$$\ts |E_H \cap E_C| \geq \frac{|E_C|}{2}.$$
Since, by Proposition~\ref{prop: known properties and values}, $\reg I(X_C) = \frac{|E_C|}{2}$,
we get 
$$
\reg K[E_C]/I(X_C) \leq |E_H \cap E_C|-1.
$$
We deduce, by Proposition~\ref{prop: computing reg by reducing to Artinian quotient}, that any 
monomial in $K[E_C]$ of degree $|E_H \cap E_C|$, belongs to the ideal $(I(X_C), t_{e})$.
Let $\t^\beta$ be the monomial in $K[E_C]$ given by the multiplication of the variables corresponding to 
the edges of $E_H \cap E_C$. Then,  
$\t^\beta \in (I(X_C), t_{e})$. Since,
by Corollary~\ref{cor: relation of the ideal with the ideal with respect to a subgraph}, 
$I(X_C)=I(X_G) \cap K[E_C]$ and $\t^\beta | \, \t^\alpha$, we conclude that
$\t^\alpha \in (I(X_G), t_{e})$, as desired. 

\noindent
In the second case, $e \notin E_C$. Then, from \eqref{eq: L957}, we get 
\begin{equation}\label{eq: L820}
\ts|E_H \cap E_C| \geq \frac{|E_C|}{2} + 1 = \reg K[E_C]/I(X_C) +2.
\end{equation}
Consider the graph $C'=C \cup \set{e}\subseteq G$. 
If $C$ and $\set{e}$ have two vertices in common, then 
$C'$ is Hamiltonian and, by Proposition~\ref{prop: known properties and values},
$$
\reg K[E_{C'}] / I(X_{C'}) = \reg K[E_{C}] / I(X_{C}).
$$ 
If $C$ and $\set{e}$ have either no vertex in common or just one vertex in common, 
then they are blocks of $C'$ and, by Proposition~\ref{prop: known properties and values},
$$
\reg K[E_{C'}] / I(X_{C'}) = \reg K[E_{C}] / I(X_{C})+1.
$$
In both cases, using \eqref{eq: L820}, we get: 
$$
|E_H \cap E_C| \geq {\reg} K[E_{C'}]/I(X_{C'}) +1.
$$ 
By Proposition~\ref{prop: computing reg by reducing to Artinian quotient}, this implies that 
any monomial in $K[E_{C'}]$ of degree \mbox{$|E_H \cap E_C|$}, belongs to 
the ideal $(I(X_{C'}), t_e)$.
Let $\t^\beta\mid \t^\alpha$ be the monomial given by the multiplication of the variables corresponding to 
the edges of $E_H \cap E_C$. 
Then 
$$
\t^\beta \in (I(X_{C'}), t_e) \subseteq (I(X_G), t_e),
$$ 
and thus $\t^\alpha \in (I(X_G), t_e)$, as we wanted.
\end{proof}

\begin{rmk}
Proposition~\ref{prop: general upper bound for reg} and Theorem~\ref{thm: reg for bipartites is max join} yield 
$\mu(G) \leq |V_G|-b_0(G)$, for any bipartite graph. A better bound can be achieved using \cite[Corollary~3.5]{SoZa93} 
and the additivity of $\reg I(X_G)$ along connected components. 
Let $G_1,\dots,G_r$ be the connected components of a bipartite graph and,
for each $i$, let $c_i$ be the length of the longest circuit in $G_i$, i.e., the \emph{circumference} of $G_i$. Then
$$
\ts \reg I(X_G) = \mu(G) \leq |V_G| - \sum_i \frac{c_i}{2}.
$$
\end{rmk}

\subsection{Regularity and Nested Ear decompositions} 
The notion of ear decomposition of a graph is involved in
Whitney's Theorem, which states that their existence 
is equivalent to the $2$-connectedness of the graph. 
Let us recall the definition of ear decomposition. 

\begin{definition}
An ear decomposition of $G$ is of a collection of $r>0$ subgraphs
$\P_0$, $\P_1,\dots,\P_r$, 
the edge sets of which form a partition of $E_G$, such that 
$\P_0$ is a vertex and, for all $1\leq i \leq r$, $\P_i$ 
is a path with end-vertices in $\P_0\cup \cdots \cup \P_{i-1}$ and with \emph{none} 
of its inner vertices in $\P_0\cup \cdots \cup \P_{i-1}$.
\end{definition}

The paths $\P_1,\dots, \P_r$ are called \emph{ears} of the decomposition. Their number, for distinct decompositions
of a graph, does not change, as each new ear increases the genus of the construction by one. 
However the number of ears of even length, and therefore the number of odd length ears, 
may change, as we show in the following example.

\begin{examp}\label{examp: a non nested ear decomp}
Consider the graph, $G$, in Figure~\ref{fig: non nested ear decomp}. It is a Hamiltonian bipartite graph, 
with Hamiltonian cycle $(1,2,5,4,3,6,1)$. 
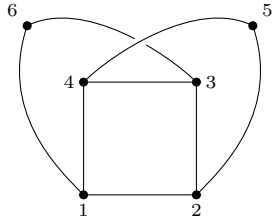
\begin{figure}[ht]
\begin{center}
\begin{tikzpicture}[line cap=round,line join=round, scale=1.5]
\draw [fill=black] (-.5,0) circle (1pt) node[anchor = north] {$\ss 1$};
\draw [fill=black] (.5,0) circle (1pt) node[anchor = north] {$\ss 2$};
\draw [fill=black] (-.5,1) circle (1pt) node[anchor = east] {$\ss 4$};
\draw [fill=black] (.5,1) circle (1pt) node[anchor = west] {$\ss 3$};
\draw [fill=black] (-1,1.5) circle (1pt) node[anchor = south east] {$\ss 6$};
\draw [fill=black] (1,1.5) circle (1pt) node[anchor = south west] {$\ss 5$};

\draw (-.5,0) -- (.5,0);
\draw (-.5,1) -- (.5,1);
\draw (-.5,0) -- (-.5,1);
\draw (.5,0) -- (.5,1);
\draw (-1,1.5) .. controls (-1.25,.75) and (-.75,.25).. (-.5,0); 
\draw (-1,1.5) .. controls (-.5,1.75) and (.25,1.25).. (.5,1); 
\draw [white, fill=white] (0,1.35) circle (1.5pt) ;
\draw (1,1.5) .. controls (.5,1.75) and (-.25,1.25).. (-.5,1); 
\draw (1,1.5) .. controls (1.25,.75) and (.75,.25).. (.5,0); 
\end{tikzpicture}
\end{center}
\caption{A Hamiltonian bipartite graph.}
\label{fig: non nested ear decomp} 
\end{figure}
This cycle can be taken 
as the ear $\P_1$ of an ear decomposition starting from $\P_0=1$. The remaining ears are all edges, $\P_2=\set{2,3}$
and $\P_3=\set{1,4}$, for example, in this order. Another ear decomposition of $G$ is given by:
$$
\P_0=1,\quad \P_1=(1,2,3,4,1),\quad \P_2=(2,5,4), \quad \P_3=(3,6,1).
$$
Whereas the first ear decomposition has a single even length ear, the second has three.
\end{examp}

\begin{definition}
The minimum number of even length ears in an ear decomposition of a graph $G$ is denoted by $\varphi(G)$.  
\end{definition}
This definition was given in \cite{Fr93} and is related to L\`ovasz characterization of factor-critical graphs. 
In \cite[Theorem~4.5]{Fr93} it is shown that for a connected graph,
\begin{equation}\label{eq: relation between mu and phi}
\ts \mu(G) = \frac{\varphi(G) + |V_G|-1}{2}.
\end{equation}

A subclass of the class of graphs endowed with an ear decomposition, i.e., by Whitney's Theorem, a subclass of 
the class of $2$-connected graphs, consists of those graphs 
that admit a special type of ear decomposition, called \emph{nested ear decomposition}. 
This definition was given in \cite{Ep92} where it was shown that this class consists of all two-terminal series parallel graphs. 
They are interesting because of the recent work \cite{Ne} on the regularity of vanishing ideals. Let us recall the definition
of nested ear decomposition.

\begin{definition}\label{def: nested decompositions of a graph}
Let $\P_0$, $\P_1,\dots,\P_r$ be an ear decomposition of a graph, $G$. 
If a path $\P_i$ has both its end-vertices in $\P_j$ 
we say that $\P_i$ is nested in $\P_j$ and we define the corresponding \emph{nest interval} to be the subpath of $\P_j$ 
determined by the end-vertices of $\P_i$. 
An ear decomposition of $G$ is \emph{nested} if, 
for all $1\leq i\leq r$, the path $\P_i$ is nested in a previous subgraph of the decomposition, $\P_j$, with 
$j<i$, and, in addition,  
if two paths $\P_i$ and $\P_l$ are nested in $\P_j$, then either  
the correspon\-ding nest intervals in $\P_j$ have disjoint edge sets or one is contained in the other.
\end{definition}

It is easy to construct graphs endowed with nested ear decompositions. One can check that none of the 
ear decompositions given in Example~\ref{examp: a non nested ear decomp} is nested. 
In \cite[Theorem~4.4]{Ne} it is shown that the regularity of the vanishing ideal over a bipartite graph endowed with 
a nested ear decomposition with $\epsilon$ even ears is a function of $|V_G|$, $\epsilon$ and the order of the field.
From this result we derive the following.

\begin{prop}
Let $G$ be a bipartite graph endowed with nested ear decomposition with $\epsilon$ even ears. Then 
$\reg I(X_G) = \frac{|V_G|+\epsilon -1}{2}$.
\end{prop}

\begin{proof}
By Proposition~\ref{prop: regularity is independent of the field} and Proposition~\ref{prop: Link to vanishing ideals}
we may take the value of the regularity, given in \cite[Theorem~4.4]{Ne}, for the field $K=\ZZ_3$, i.e.,
setting $q=3$. To the result we add $1$, since in \cite{Ne} $\reg G$  is the regularity
of the quotient of the polynomial ring by the vanishing ideal over $G$.
\end{proof}

If follows from this result that for a bipartite graph endowed with a nested ear decomposition
the number of ears of even length does not change among all nested ear decompositions of the graph. This conclusion was already drawn 
in \cite[Corollary~4.5]{Ne}. But now, by Theorem~\ref{thm: reg for bipartites is max join}, we know that 
$\reg I(X_G) = \mu(G)$. This, together with  \eqref{eq: relation between mu and phi},
yields the following stronger result. 

\begin{cor}\label{cor: new combinatorial corollary}
If $G$ is a bipartite graph endowed with a nested ear decomposition, then the number of even length ears
in any such decomposition coincides with $\varphi(G)$, the minimum number of even length ears in any ear decomposition
of $G$.
\end{cor}

\end{document}